\documentclass[a4paper,11pt, oneside]{amsart}
\title{A Borel--Weil Theorem for the Irreducible Quantum Flag Manifolds}
\usepackage[english]{babel}
\usepackage{amsthm}
\usepackage{amssymb}
\usepackage{amsfonts}
\usepackage{amsmath}
\usepackage{amsthm}
\usepackage[all]{xy}
\usepackage{geometry}
\usepackage{ae,aecompl}
\usepackage{eurosym}
\usepackage{verbatim}
\usepackage{mathrsfs}
\usepackage{caption}
\usepackage{url}
\usepackage{tikz}
\usepackage{hyperref}

\newcounter{stepcounter}
\newtheoremstyle{smallcaps}
  {3pt}          
  {3pt}          
  {\itshape}          
  {}              
  {\sc}          
  {.}             
  {.5em}            
  {} 
  
\newtheoremstyle{smallcapsdef}
  {3pt}          
  {3pt}          
  {}          
  {}              
  {\sc}          
  {.}             
  {.5em}            
  {} 

\theoremstyle{plain}

\newtheorem{thm}{Theorem}[section]

\newtheorem{lem}[thm]{Lemma}
\newtheorem{prop}[thm]{Proposition}
\newtheorem{cor}[thm]{Corollary}

\theoremstyle{definition}

\newtheorem{defn}[thm]{Definition}

\newtheorem{remark}[thm]{Remark}

\date{}

\newcommand\bit{\begin{itemize}}
\newcommand\eit{\end{itemize}}
\newcommand\bet{\begin{enumerate}}
\newcommand\eet{\end{enumerate}}
\newcommand\ed{\end{document}}

\DeclareFontFamily{U}{mathx}{\hyphenchar\font45}
\DeclareFontShape{U}{mathx}{m}{n}{
   <5> <6> <7> <8> <9> <10>
   <10.95> <12> <14.4> <17.28> <20.74> <24.88>
   mathx10
   }{}
\DeclareSymbolFont{mathx}{U}{mathx}{m}{n}
\DeclareFontSubstitution{U}{mathx}{m}{n}
\DeclareMathAccent{\widecheck}{0}{mathx}{"71}
\DeclareMathAccent{\wideparen}{0}{mathx}{"75}

\renewcommand{\a}{\alpha}
\renewcommand{\d}{\delta}
\newcommand{\e}{\varepsilon}

\newcommand\Om{\Omega}

\newcommand\del{\partial}
\newcommand\adel{\ol{\partial}}
\newcommand\DEL{\Delta}

\newcommand\bC{{\mathbb C}}

\newcommand\bR{{\mathbb R}}
\newcommand\bZ{{\mathbb Z}}

\newcommand\F{{\mathcal F}}

\renewcommand{\O}{\mathcal{O}}

\newcommand\can{\mathrm{can}}

\newcommand\co{\mathrm{co}}

\newcommand\exd{\mathrm{d}}

\newcommand\hw{\mathrm{hw}}

\newcommand\id{\mathrm{id}}
\newcommand\proj{\mathrm{proj}}

\newcommand\inv{^{-1}}

\newcommand\oby{\otimes}

\newcommand\wed{\wedge}
\newcommand\sseq{\subseteq}

\def\qbinom#1#2{\ensuremath{\left[\kern-.3em\left[\genfrac{}{}{0pt}{}{#1}{#2}\right]\kern-.3em\right]_q}}

\newcommand\ol{\overline}

\newcommand\mto{\mapsto}

\newcommand\algn{algebra}

\newcommand\EE{{\mathcal E}}

\newcommand{\OO}{\mathcal{O}}
\newcommand{\qphs}{\OO_q(G/L^{\mathrm{s}}_S)}
\newcommand{\sslevi}{\OO_q(L^{\mathrm{s}}_S)}
\newcommand{\usslevi}{U_q(\mathfrak{l}^{\,\mathrm{s}}_S)}

\newcommand{\ssL}{L^{\,\mathrm{s}}_S}
\newcommand{\uslevi}{U_q(\mathfrak{l}_S)}

\usepackage{tikz}
\usetikzlibrary{decorations.pathreplacing}

\usepackage{ytableau}

\author[A. Carotenuto]{Alessandro Carotenuto}
\address{Mathematical Institute of Charles University, Sokolovsk\'a 83, Prague, Czech Republic}
\email{acaroten91@gmail.com}

\author[F. D\'iaz Garcia]{Fredy D\'iaz Garc\'ia}
\address{Mathematical Institute of Charles University, Sokolovsk\'a 83, Prague, Czech Republic} \email{diazf@karlin.mff.cuni.cz}

\author[R. \'O Buachalla]{R\'eamonn \'O Buachalla}
\address{Mathematical Institute of Charles University, Sokolovsk\'a 83, Prague, Czech Republic} \email{obuachalla@karlin.mff.cuni.cz}

\usepackage[english]{babel}

\thanks{R\'OB, AC, and FDG are supported by the Charles University PRIMUS grant \emph{Spectral Noncommutative Geometry of Quantum Flag Manifolds} PRIMUS/21/SCI/026. AC was also supported by the GA\v{C}R project 20-17488Y and \mbox{RVO: 67985840}. FDG was also supported by Conacyt (Consejo Nacional de Ciencia y Tecnología, M\'exico) and PAEP (Programa de apoyos para estudios de posgrado, UNAM, M\'exico). R\'OB also acknowledges FNRS support through a postdoctoral fellowship within the framework of the MIS Grant ``Antipode'' grant number F.4502.18, and was also supported by GA\v{C}R through the framework of the grant GA19-06357S.}

\begin{document}

\maketitle

\begin{abstract}
We establish a noncommutative generalisation of the Borel--Weil theorem for the Heckenberger--Kolb calculi of the irreducible quantum flag manifolds $\OO_q(G/L_S)$, generalising previous work of a number of authors (including the first and third authors of this paper) on the quantum Grassmannians $\OO_q(\mathrm{Gr}_{n,m})$. As a direct consequence we get a novel noncommutative differential geometric presentation of the quantum coordinate rings $S_q[G/L_S]$ of the irreducible quantum flag manifolds. 
The proof is formulated in terms of quantum principal bundles, and the recently introduced notion of a principal pair, and uses the Heckenberger and Kolb first-order differential calculus for the quantum Possion homogeneous spaces $\OO_q(G/L^{\,\mathrm{s}}_S)$.
\end{abstract}

\tableofcontents

\section{Introduction}

The classical Borel--Weil theorem is a foundational result in geometric representation theory which realises each irreducible representation of a complex semisimple Lie algebra $\mathfrak{g}$ as the space of holomorphic sections of a  line bundle over a flag manifold. Extensions of this result to the setting of quantum groups came soon after the discovery of quantum groups themselves. These generalisations took a variety of different forms \cite{APW,PW,JuSt,GZ,noumi0,NoumiUn}, as discussed for example, in the introduction to \cite{KMOS}.


In all of the above works no explicit noncommutative formulation of holomorphicity was used. In recent years, however, a framework for describing noncommutative holomorphic sections has emerged, expressed in terms of differential graded algebras and generalising the Koszul--Malgrange presentation of holomorphic vector bundles \cite{KM}. Indeed, this approach was in part developed with a view to establishing a differential geometric \mbox{$q$-deformation} of the Borel--Weil theorem. This direction of research was initiated by Majid in his influential paper on the Podle\'s sphere \cite{Maj}, see also \cite[\textsection 7.4.1]{BeggsMajid:Leabh}. It was continued by Khalkhali, Landi, van Suijlekom, and Moatadelro in \cite{KLVSCP1,KKCP2, KKCPN} where the definitions of complex structure and noncommutative holomorphic vector bundle were introduced and the family of examples extended to include quantum projective space $\OO_q(\mathbb{CP}^{n})$. An equivalent definition of holomorphic structure would later appear independently in the work of Beggs and Smith on noncommutative coherent sheaves \cite{BS}.


The differential graded algebras underlying all of these constructions are the remarkable Heckengerber--Kolb differential calculi, arguably the most important family of noncommutative differential structures in the theory of quantum groups. These calculi are an essentially unique covariant $q$-deformation of the de Rham complex of the irreducible quantum flag manifolds. This family contains the quantum Grassmannians as the \mbox{$A$-series} special case. In particular, it contains quantum projective space $\OO_q(\mathbb{CP}^n)$, which reduces to the celebrated Podle\'s sphere for the $n=1$ case. Thus it is very natural to try and extend the Borel--Weil theorem from quantum projective space to the general irreducible quantum flag manifolds. A significant step in this direction was the proof, by the second and third authors and Krutov, Somberg, and Strung in \cite{DOKSS}, that each relative line module over an irreducible quantum flag manifold admits a unique covariant holomorphic structure. Building on this work, the Borel--Weil theorem was generalised to the quantum Grassmannians by Morzinski and the first and third authors in \cite{KMOS}. A major element of the proof was the construction of a differential calculus, extending the Heckenberger--Kolb calculus quantum Grassmannian calculus, for the Poisson quantum homogeneous space $\OO_q(S^{n,m})$, that is to say, the direct sum of the line modules over $\O_q(\mathrm{Gr}_{n,m})$. This was achieved by taking a suitable quotient of the bicovariant calculus of $\OO_q(SU_n)$ associated to its coquasitriangular structure. This extended the work of the third author in \cite{MMF1} for the special case of quantum projective space $\OO_q(\mathbb{CP}^{n})$, and provided a novel coquasitriangular presentation of the Heckenberger--Kolb calculus of $\OO_q(\mathrm{Gr}_{n,m})$. 


While this approach is almost certainly extendable to the general irreducible setting, it is at present not clear how to do so without a concerted technical effort. Hence in this paper we adopt a different approach and take advantage of the work of Heckenberger and Kolb who already constructed a left $\OO_q(G)$-covariant first-order differential calculus over $\OO_q(G/L^{\,\mathrm{s}}_S)$, the quantum Poisson homogeneous space of $\OO_q(G/L_S)$. We show that these calculi give a quantum principal bundles over the irreducible quantum flag manifolds using the formal framework of principal pairs, as introduced by Mrozinski and the first and third authors in \cite{KMOS}. Principal pairs, a special type of principal comodule algebra constructed from a nested pair of quantum homogeneous spaces, and provide a simple framework in which to verify the requirements of a quantum principal bundle.

We then build upon this quantum principal presentation and construct principal connections. The Heckenberger--Kolb calculus over $\OO_q(G/\ssL)$ deforms a certain subspace of the complexified one-forms of $G/\ssL$, namely the subspace whose Lie algebra is spanned by positive root vectors. This means the bundle has no horizontal forms, and that the zero map is a strong principal connection. This results in a particularly simple realisation of the holomorphic structures $\adel_{\EE_k}:\EE_k \to \Omega^{(0,1)} \otimes_{\OO_q(G/L_S)} \EE_k$. Indeed, it enables us to simplify the approach of \cite{KMOS}, and to finally prove the Borel--Weil theorem for \emph{all} the irreducible quantum flag manifolds.

  
An important motivation of the paper is to further explore the connections between the quantum coordinate algebras $\OO_q(G/L_S)$ and their quantum homogeneous coordinate ring counterparts $S_q[G/L_S]$ studied in noncommutative projective geometry, see for example \cite{Avdb,RigalZadunaisky,Soib,TaftTowber}. For the special case of the quantum Grassmannians, these rings are $q$-deformations of the Pl\"ucker embedding homogeneous coordinate ring and are important examples in the theory of quantum cluster algebras \cite{KKCP2,KKCPN}. The quantum Borel--Weil theorem gives us a direct $q$-deformation of the classical ample line bundle presentation of $S_q[G/L_S]$. This directly generalises the work of \cite{KLVSCP1,KKCP2, KKCPN} for quantum projective space, and gives us an important point of contact between noncommutative differential geometry and noncommutative projective geometry.

The Borel--Weil theorem has a number of important applications in associated works. One example is to the study of the noncommutative K\"ahler geometry of $\OO_q(G/L_S)$. The cohomological information given by the Borel--Weil theorem allows us to identify which line modules over $\OO_q(G/L_S)$ are positive and which are negative. Through an application of the Kodaira vanishing theorem \cite{OSV} we can then establish vanishing of higher cohomologies for positive line modules. Moreover, the Borel--Weil theorem allows us to conclude that twisting the Dolbeault--Dirac operator of $\OO_q(G/L_S)$ by a negative line module produces a Fredholm operator, which is to say, it allows us to conclude analytic behaviour from purely geometric information \cite{DOSFred}.

  

\subsection{Summary of the Paper}

~~~~

In \textsection 2, we recall the basic notions of Hopf--Galois extensions and quantum homogeneous spaces. We recall how such structures interact with first-order differential calculi, focusing on the theory of quantum principal bundles. 


In \textsection 3 we treat the Drinfeld--Jimbo quantised enveloping algebras, quantum coordinate algebras, and the quantum flag manifolds. 


In \textsection 4 we recall the well-known tangent space formulation of covariant first-order calculi over quantum homogeneous $\OO_q(G)$-spaces. We then use the tangent space formulation of the Heckenberger--Kolb calculi to establish a direct $q$-deformation of Liouville's theorem for the irreducible quantum flags, which is to say, we show that the first anti-holomorphic cohomology group of the calculus is of dimension $1$.


In \textsection 5 we give a quantum principal bundle description of the Heckenberger--Kolb calculus, construct a principal connection, and use it to give an explicit realisation of the covariant holomorphic structures of the relative line modules $\EE_k$ over $\OO_q(G/L_S)$. 


In \textsection 6 we establish the main result of the paper, namely the Borel--Weil theorem for the irreducible quantum flag manifolds.

\begin{thm}
For any irreducible quantum flag manifold $\OO_q(G/L_S)$, it holds that 
\begin{enumerate}
  \item $H^0(\mathcal{E}_k)$ is an irreducible $U_q(\mathfrak{g})$-module of highest weight $-k\varpi_{x}$,
  \item $H^0(\mathcal{E}_{-k}) = 0$.
\end{enumerate}
for all $k \in \mathbb{Z}_{>0}$, where $\Pi\backslash S$ consists of the simple root $\alpha_x$, and $\varpi_x$ is the corresponding fundamental weight.
\end{thm}

We then present an application of the Borel--Weil theorem and give a novel noncommutative differential geometric presentation of the quantum homogeneous coordinate rings $S_q(G/L_S)$ of the irreducible quantum flag manifolds. 
We finish by describing the situation for the opposite complex structure.


\subsection*{Acknowledgments} We would like to thank to Elmar Wagner for suggesting the proof of Theorem \ref{thm:Liouville}. 

\section{Preliminaries on Quantum Principal Bundles}

We begin with a presentation of the necessary results on quantum group noncommutative geometry, namely the theory of covariant differential calculi and quantum principal bundles. All this material is by now quite well-know, and a more detailed presentation can be found in the recent monograph \cite{BeggsMajid:Leabh}.

\subsection{First-Order Differential Calculi}

A {\em first-order differential calculus}, or simply a \emph{fodc}, over an algebra $B$ is a pair $(\Om^1,\exd)$, where $\Omega^1$ is a $B$-bimodule and $\exd: B \to \Omega^1$ is a derivation such that $\Om^1$ is generated as a left $B$-module by those elements of the form~$\exd b$, for~$b \in B$. We call $\exd$ the \emph{exterior derivative} of the fodc. The {\em universal fodc} over $B$ is the pair
$(\Om^1_u(B), \exd_u)$, where $\Om^1_u(B)$ is the kernel of the multiplication map $m_B: B \otimes B \to B$ endowed
with the obvious bimodule structure, and $\exd_u$ is the map defined by
\begin{align*}
\exd_u: B \to \Omega^1_u(B), & & b \mto 1 \otimes b - b \otimes 1.
\end{align*}
Every fodc over $B$ is of the form $\left(\Omega^1_u(B)/N, \,\proj \circ \exd_u\right)$, where $N$ is a $B$-sub-bimodule of $\Omega^1_u(B)$, and 
$
\proj:\Omega^1_u(B) \to \Omega^1_u(B)/N
$
is the canonical quotient map. For any subalgebra $B' \sseq B$, the \emph{restriction} of a fodc over $B$ to $B'$ is the fodc $\Omega^1(B') \sseq \Omega^1(B)$ over $B'$ generated by the elements $\mathrm{d}b$, for $b \in B'$.

For $H$ a Hopf algebra $H$, and $P$ a right $H$-comodule algebra, a fodc over $P$ is said to be {\em right $H$-covariant} if the following (necessarily unique) map is well defined
\begin{align*}
\DEL_R: \Om^1(P) \to \Om^1(P) \otimes H, & & p \exd q \mto p_{(0)} \exd q_{(0)} \otimes p_{(1)} q_{(1)}.
\end{align*} 
Covariance for a fodc over a left comodule algebra is defined similarly. 
%

\subsection{Connections and Line Modules}

For $\Omega^1$ a fodc over an algebra $B$, and $\mathcal{F}$ a left $B$-module, a \emph{connection} on $\F$ is a $\mathbb{C}$-linear map $\nabla:\mathcal{F} \to \Omega^1 \otimes_B \F$ satisfying 
\begin{align*}
\nabla(bf) = \exd b \otimes f + b \nabla f, & & \textrm{ for all } b \in B, f \in \F.
\end{align*}
In this paper, we are concerned not with connections for general $B$-modules, but for modules which generalise the space of sections of a classical line bundle: a \emph{line module} over $B$ is be an invertible $B$-bimodule $\EE$, where \emph{invertible} means that there exists another $B$-bimodule $\EE^{\vee}$ such that 
\begin{align} \label{eq:linemodules}
\EE \otimes_B \EE^{\vee} \simeq \EE^{\vee}\! \otimes_B \EE \simeq B.
\end{align}
Note that any such $\EE$ is automatically projective as a left, and right, $B$-module. 


\subsection{Hopf--Galois Extensions} \label{subsection:PCA}

A right $H$-comodule algebra $(P,\Delta_R)$ is said to be an {\em $H$-Hopf--Galois extension of} $B := P^{\co(H)}$ 
if for $m_P:P \otimes_B P \to P$ the multiplication of~$P$,
a bijection is given by 
$$
\can := (m_P \otimes \id) \circ (\id \otimes \DEL_R): P \otimes_B P \to P \otimes H.
$$
The Hopf--Galois condition is equivalent to exactness of the sequence
\begin{align} \label{eqn:qpbexactseq}
0 \longrightarrow P\Om^1_u(B)P {\buildrel \iota \over \longrightarrow} \Om^1_u(P) {\buildrel {\mathrm{ver}}\over \longrightarrow} P \oby H^+ \longrightarrow 0,
\end{align}
where $\iota$ is the inclusion map, and 
$
\mathrm{ver} := \mathrm{can} \circ \mathrm{proj}_B
$
with $\proj_B$ the restriction to $\Omega^1_u(P)$ of the canonical projection $P \otimes P \to P \otimes_B P$. 


We now consider a special type of Hopf--Galois extension which plays an important role in this paper. Let $A$ be a Hopf algebra, then a \emph{quantum homogeneous space} is a coideal subalgebra $B \subseteq A$ such that $A$ is faithfully flat as a right $B$-module, and $AB^+ = B^+A$, where we have written $B^+ := \ker(\e|_B: B \to \mathbb{C})$. Denoting by $\pi_B:A \to A/B^+A $ the canonical surjective Hopf algebra map, the associated coaction
$$
\Delta_{R,\pi_B(A)} := (\id \otimes \pi_B) \circ \Delta : A \to A \otimes \pi_B(A),
$$
gives $A$ the structure of a right $\pi_B(A)$-comodule algebra. Moreover, as shown in \cite{Tak}, faithful flatness implies that $B$ is the space of coinvariants of this coaction. As is well-known, $P$ is automatically a Hopf--Galois extension of $B$, see for example \cite[Lemma 3.9]{T72}. A \emph{relative line module} over a quantum homogeneous $A$-space $B$ is a $B$-sub-bimodule, left $A$-comodule, $\EE \subseteq A$ that is also a line module over $B$, and for which the isomorphisms in \eqref{eq:linemodules} are left $A$-comodule maps.

\subsection{Quantum Principal Bundles}\label{subsection:QPB}

The following definition, due to Brzezi\'nski and Majid \cite{BeggsMajid:Leabh,TBSM1}, presents sufficient criteria for the existence of a non-universal version of the sequence \eqref{eqn:qpbexactseq}.

\begin{defn} \label{qpb}
Let $H$ be a Hopf algebra. A {\em quantum principal $H$-bundle} is a pair $(P,\Omega^1(P))$, consisting of a right $H$-comodule algebra $(P,\Delta_R)$ and a right-$H$-covariant calculus $\Om^1(P)$, such that:
\begin{enumerate}
  \item $P$ is a Hopf--Galois extension of $B = P^{\,\co(H)}.$
  \item If $N \sseq \Om^1_u(P)$ is the sub-bimodule of the universal calculus corresponding to $\Om^1(P)$, we have $\mathrm{ver}(N) = P \otimes I$, for some $\mathrm{Ad}$-sub-comodule right ideal 
\[ I \sseq H^+ := \ker(\e: H \to \mathbb{C}),
\] where $
\mathrm{Ad} : H \to H \otimes H$ is defined by $\mathrm{Ad}(h) := h_{(2)} \otimes S(h_{(1)}) h_{(3)}.$
\end{enumerate}
\end{defn}

Denoting by $\Om^1(B)$ the restriction of $\Om^1(P)$ to $B$, and $\Lambda^1(H) := H^+/I$, the quantum principal bundle definition implies that an exact sequence is given by
\begin{align} \label{Eqn:qpbexactseq}
0 \longrightarrow P\Om^1(B)P {\buildrel \iota \over \longrightarrow} \, \Om^1(P) {\buildrel {~\mathrm{ver}~~}\over \longrightarrow} P \otimes \Lambda^1(H) \longrightarrow 0,
\end{align}
where by abuse of notation $\mathrm{ver}$ denotes the map induced on $\Om^1(P)$ by identifying~$\Om^1(P)$ as a quotient of $\Om_u^1(P)$. We denote $\Omega^1(P)_{\mathrm{hor}} := P\Om^1(B)P$ and call this subspace the \emph{horizontal forms} of the bundle.


\subsection{Principal Connections} \label{subsection:ConnFromPrinCs}

In this subsection, we briefly recall the theory of principal connections for quantum principal bundle at a level of generality suitable to the paper. For the full picture see \cite[\textsection 5]{BeggsMajid:Leabh}, or the accompanying paper \cite[\textsection 2]{KMOS}.

\begin{defn}
A {\em principal connection} for a quantum principal $H$-bundle $(P,\Omega^1(P))$ is a left $P$-module, right $H$-comodule, projection $\Pi:\Om^1(P) \to \Om^1(P)$ satisfying
$$
\ker(\Pi) = P\Om^1(B)P.
$$
A principal connection $\Pi$ is called {\em strong} if $(\id - \Pi) \big(\exd P\big) \sseq \Om^1(B)P$. We denote $\Omega^1(P)_{\mathrm{ver}} := \mathrm{im}(\Pi)$ and call this subspace the space of \emph{vertical forms} associated to $\Pi$. 
\end{defn}

For any left $P$-submodule right $H$-subcomodule $\EE \subseteq P$, we can use principal connections to define a connection $\nabla:\EE \to \Omega^1(B) \otimes_B \EE$: 
Note first that an isomorphism 
\begin{align*}
j:\Om^1(B) \oby_B \EE \hookrightarrow \Om^1(B)\EE,
\end{align*}
is given by the multiplication map.
A strong principal connection $\Pi$ defines a connection $\nabla$ on $\EE$ by
\begin{align*}
 \nabla: = j^{-1} \circ (\id - \Pi) \circ \exd: \EE \to \Om^1(B) \oby_B \EE.
\end{align*}
Consider now the special case where $P$ is endowed with a left $A$-coaction giving it the structure of an $(A,H)$-bicomodule. If the principal connection $\Pi$ is a left $A$-comodule map, then we see that the connection $\nabla$ will also be a left $A$-comodule map. This will be the case for all the principal connections considered in this paper.

\section{Preliminaries on Drinfeld--Jimbo Quantum Groups}

In this section we treat the Drinfeld--Jimbo quantised enveloping algebras, quantum coordinate algebras, quantum flag manifolds, and their associated quantum Poisson homogeneous spaces.

\subsection{Quantised Enveloping Algebras} \label{subsection:DJQG}

Let $\mathfrak{g}$ be a finite-dimensional complex semisimple Lie algebra of rank $r$. We fix a Cartan subalgebra $\mathfrak{h}$ and choose a set of simple roots $\Pi = \{\alpha_1, \dots, \alpha_\ell\}$ for the corresponding root system in the linear dual of $\mathfrak{g}$. We denote by $(\cdot,\cdot)$ the symmetric bilinear form induced on $\mathfrak{h}^*$ by the Killing form of $\mathfrak{g}$, normalised so that any shortest simple root $\a_i$ satisfies $(\a_i,\a_i) = 2$. The Cartan matrix $(a_{ij})$ of $\mathfrak{g}$ is defined by 
$
a_{ij} := \big(\alpha_i^{\vee},\alpha_j\big),
$
where $\alpha_i^{\vee} := 2\a_i/(\a_i,\a_i)$.

Let $q \in \bR$ such that $q \neq -1,0,1$, and denote $q_i := q^{(\alpha_i,\alpha_i)/2}$. The {\em Drinfeld--Jimbo quantised enveloping \algn} $U_q(\mathfrak{g})$ is the noncommutative associative algebra generated by the elements $E_i, F_i, K_i$, and $K^{-1}_i$, for $ i=1, \ldots, l$, subject to the relations 
\begin{align*}
 K_iE_j = q_i^{a_{ij}} E_j K_i, ~~~~ K_iF_j= q_i^{-a_{ij}} F_j K_i, ~~~~ K_i K_j = K_j K_i, ~~~~ K_iK_i^{-1} = K_i^{-1}K_i = 1,\\
 E_iF_j - F_jE_i = \d_{ij}\frac{K_i - K\inv_{i}}{q_i-q_i^{-1}}, ~~~~~~~~~~~~~~~~~~~~~~~~~~~~~~~~~~~~~~~~~
\end{align*}
and the quantum Serre relations which we omit (see \cite[\textsection 6.1.2]{KSLeabh} for details).
The formulae 
\begin{align*}
\DEL(K_i) = K_i \oby K_i, & & \DEL(E_i) = E_i \oby K_i + 1 \oby E_i, & & \DEL(F_i) = F_i \oby 1 + K_i\inv \oby F_i 
\end{align*}
define a Hopf algebra structure on $U_q(\mathfrak{g})$, satisfying $\e(E_i) = \e(F_i) = 0$, and $\e(K_i) = 1$. 

\subsection{Type-1 Representations}

The set of {\em fundamental weights} $\{\varpi_1, \dots, \varpi_l\}$ of $\mathfrak{g}$ is the dual basis of simple coroots $\{\a_1^{\vee},\dots, \a_l^{\vee}\}$. 
We denote by ${\mathcal P}$ the $\bZ$-span of the fundamental weights, and by ${\mathcal P}^+$ the $\mathbb{Z}_{\geq 0}$-span of the fundamental weights. 
For any $U_q(\mathfrak{g})$-module $V$, a vector $v\in V$ is called a \emph{weight vector} of \emph{weight}~$\mathrm{wt}(v) \in \mathcal{P}$ if
\begin{align}\label{eq:Kweight}
K_i \triangleright v = q^{(\alpha_i,\mathrm{wt}(v))} v, & & \textrm{ for all } i=1,\ldots,r.
\end{align}
For each $\lambda \in\mathcal{P}^+$ there exists an irreducible finite-dimensional $U_q(\mathfrak{g})$-module $V_\lambda$, uniquely defined by the existence of a weight vector $v_{\mathrm{hw}}\in V_\lambda$ of weight $\lambda$, which we call a {\em highest weight vector}, satisfying
$
E_i \triangleright v_{\mathrm{hw}}=0,
$
for all $i = 1, \dots, l$. We call any finite direct sum of such $U_q(\mathfrak{g})$-representations a {\em type-$1$ representation}. 
Each type-$1$ module $V_{\lambda}$ decomposes into a direct sum of \emph{weight spaces}, which is to say, subspaces of $V_{\lambda}$ spanned by weight vectors of any given weight.

We denote by $_{U_q(\mathfrak{g})}\mathbf{type}_1$ the full subcategory of ${U_q(\mathfrak{g})}$-modules whose objects are finite sums of type-$1$ modules $V_\lambda$, for $\lambda \in \mathcal{P}^+$. 
This category admits a braided monoidal structure. Explicitly, for $V$ and $W$ two finite-dimensional irreducible objects in $_{U_q(\mathfrak{g})}\mathbf{type}_1$, the braiding is completely determined if one demands that $R_{V,W}$ is a $U_q(\mathfrak{g})$-module homomorphism satisfying
$$
R_{V,W} (v \otimes w) = q^{(\mathrm{wt}(v),\mathrm{wt}(w))} w \otimes v + \sum_{i \in I} w_i \otimes v_i,
$$
where $I$ is some finite index set, and $w_i \in W$, $v_i \in V$ such that $\mathrm{w} \succ \mathrm{wt}(w_i)$ and $\mathrm{w} \succ \mathrm{wt}(w_i)$, with respect to the partial order $\succ$ on $\mathcal{P}$. Given a choice of weight bases $\{e_i\}_{i=1}^{\mathrm{dim}(V)}$, and $\{f_i\}_{i=1}^{\mathrm{dim}(W)}$, for two finite-dimensional $U_q(\mathfrak{g})$-modules $V, W$, 
\begin{align*}
R_{V,W}(e_i \otimes f_j) =: \sum_{k=1}^{\dim(W)} \sum_{l=1}^{\mathrm{dim}(V)} (R_{V,W})^{kl}_{ij} f_k \otimes e_l,
\end{align*}
defines the associated {\em $R$-matrix} $(R_{V,W})^{kl}_{ij}$.

\subsection{Quantum Coordinate Algebras}

Let $V$ be a finite-dimensional $U_q(\mathfrak{g})$-module, $v \in V$, and $f \in V^*$, the linear dual of $V$. Consider the function $c^{\textrm{\tiny $V$}}_{f,v}:U_q(\mathfrak{g}) \to \bC$ defined by $c^{\textrm{\tiny $V$}}_{f,v}(X) := f\big(X(v)\big)$.
%
The {\em coordinate ring} of $V$ is the subspace
\begin{align*}
C(V) := \text{span}_{\mathbb{C}}\!\left\{ c^{\textrm{\tiny $V$}}_{f,v} \,| \, v \in V, \, f \in V^*\right\} \sseq U_q(\mathfrak{g})^*.
\end{align*}
It is easily checked that $C(V)$ is contained in $U_q(\mathfrak{g})^\circ$, the Hopf dual of $U_q(\mathfrak{g})$, and moreover that a Hopf subalgebra of $U_q(\mathfrak{g})^\circ$ is given by 
\begin{align} \label{eqn:PeterWeyl}
\O_q(G) := \bigoplus_{\lambda \in \mathcal{P}^+} C(V_{\lambda}).
\end{align}
We call $\O_q(G)$ the {\em quantum coordinate algebra of $G$}, where $G$ is the compact, simply-connected, simple Lie group having $\mathfrak{g}$ as its complexified Lie algebra. Moreover, we call the decomposition of $\OO_q(G)$ given in \eqref{eqn:PeterWeyl} the \emph{Peter--Weyl decomposition} of $\OO_q(G)$.

\subsection{Quantum Flag Manifolds}

For $S$ a subset of simple roots, consider the Hopf \mbox{subalgebra} of $U_q(\mathfrak{g})$ given by 
\begin{align*}
U_q(\mathfrak{l}_S) := \big< K_i, E_j, F_j \,|\, i = 1, \ldots, l; j \in S \big>.
\end{align*} 
Consider now the coideal subalgebra of $\uslevi$-invariants
\begin{align*}
\O_q\big(G/L_S\big) := {}^{U_q(\mathfrak{l}_S)}\O_q(G),
\end{align*} 
with respect to the natural left $U_q(\mathfrak{g})$-module structure on $\OO_q(G)$. Cosemisimplicity of the Hopf dual of $\uslevi$ implies that $\OO_q(G/L_S)$ is a quantum homogeneous space. We call it the {\em quantum flag manifold associated to} $S$. 

\subsection{The Associated Quantum Poisson Homogeneous Spaces} 

Classically the algebra $\mathfrak{l}_S$ is reductive, and hence decomposes into a direct sum $\mathfrak{l}_S^{\,\mathrm{s}} \oplus \mathfrak{u}_1$, comprised of a semisimple part and a commutative part, respectively. In the quantum setting, we are thus motivated to consider the Hopf subalgebra 
\begin{align*}
\usslevi := \big< K_i, E_i, F_i \,|\, i \in S \big> \sseq U_q(\mathfrak{l}_S).
\end{align*} 
Consider now the coideal subalgebra of $\usslevi$-invariants
\begin{align*}
\O_q\big(G/L^{\,\mathrm{s}}_S\big) := {}^{\usslevi}\O_q(G).
\end{align*} 
Just as for the quantum flag manifolds, each $\O_q\big(G/L^{\,\mathrm{s}}_S\big)$ is a quantum homogeneous space. We call $\sslevi$ the \emph{quantum homogeneous Poisson space associated} to $S$. 

Let $\{v_i\}_{i=1}^{N_x}$ be a weight basis of $V_{\varpi_x}$ such that $v_{N_x}$ is a highest weight vector. We denote the dual basis of $V_{-w_0(\varpi_x)}$ by $\{f_i\}_{i=1}^{N_x}$. As shown in \cite[\textsection Theorem 4.1]{Stok}, a set of generators for $\OO_q(G/L^{\,\mathrm{s}}_S)$ is given by 
\begin{align*}
z_{i} := c^{\varpi_x}_{f_i,v_N}, & & \overline{z}_j := c^{-w_0(\varpi_x)}_{v_j,f_N} & & \text{ for } i,j = 1, \dots, N_x := \dim(V_{\varpi_x}), \, x \in \Pi\backslash S,
\end{align*}
where to lighten notation we have respectively written $\varpi_x$, and $-w_0(\varpi_x)$, as superscripts instead of $V_{\varpi_x}$, and $V_{-w_0(\varpi_x)}$.


\section{Liouville's Theorem for $\OO_q(G/L_S)$}

In this section we prove a direct $q$-deformation of Liouville's theorem for the irreducible flag manifolds. More explicitly, we show that the kernel of the anti-holomorphic differential $\adel$ of the $(0,1)$-Heckenberger--Kolb calculus contains only scalar multiples of the identity. The proof uses the tangent space formulation of a fodc and the quantum root vector presentation of the Heckenberger--Kolb calculi.

\subsection{Tangent spaces}

In this subsection we recall the tangent space formulation of covariant fodc, over quantum homogeneous spaces, as formulated in \cite{HKfodcQHS}. For sake of simplicity, we present the material for the special case of Drinfeld--Jimbo quantum groups. In particular, we consider a quantum homogeneous $\OO_q(G)$-space of invariants $B := {}^K \OO_q(G)$, for some Hopf subalgebra $K \subseteq U_q(\mathfrak{g})$. 
\begin{defn}
A finite-dimensional subspace $T \subseteq B^{\circ}$ is called a \emph{tangent space} for $B$ if 
\begin{align*}
1. ~~ \Delta(T) \subseteq T_{\e} \otimes B^{\circ}, & & 2. ~~ KT \subseteq T,
\end{align*}
where we have denoted $T_{\e} := T \oplus \mathbb{C}\e$. 
\end{defn}
For every tangent space $T$, the subspace 
\begin{align*}
I := \left\{ b \in B^+ \,|\, X(b) = 0, \mbox{for all } X \in T \right\}
\end{align*}
is a right $B$-submodule, and a right $K$-submodule, of $B^+$. Thus the quotient 
$
V := B^+/I 
$
has the structure of a right $B$-module, and a right $K$-module. 

We denote by $\Omega^1(B)$ the $K$-invariant subspace of $\OO_q(G) \otimes V$, that is to say, 
$$
\left\{ \sum_i f_i \otimes v_i \in \OO_q(G) \otimes V \,|\, \sum_i (Y \triangleright f_i)  \otimes v_i =  \sum_i  f_i \otimes (v_i  \triangleleft Y), \textrm{ \, for all } Y \in K \right\}\!.
$$
We endow $\OO_q(G) \otimes V$ with a $B$-bimodule structure by taking left multiplication on the first tensor factor, and by taking the diagonal action on the right. With respect to this choice of bimodule structure, $\Omega^1(B)$ is a $B$-sub-bimodule of $\OO_q(G) \otimes V$. Moreover, the obvious left $\OO_q(G)$-comodule structure of $\OO_q(G) \otimes V$ restricts to a left $\OO_q(G)$-coaction on $\Omega^1(B)$.

Consider next the linear map
\begin{align*}
\exd: B \to \Omega^1(B), & & b \mapsto b_{(1)} \otimes [b_{(2)}^+],
\end{align*}
where $[-]$ denotes the coset in $V$ of an element in $B$. As is readily checked, the pair $(\Omega^1(B),\exd)$ is a left $\OO_q(G)$-covariant fodc over $B$. For any choice of basis $\{X_i\}_{i=1}^n$ of $T$, an explicit formula for $\exd$ is given by 
\begin{align}\label{eqn:tangent.exd}
\exd b = \sum_{i=1}^n (X_i \triangleright b) \otimes e_i,
\end{align}
where $\{e_i\}_{i=1}^n$ is the dual basis of $V$. Moreover, as shown by Heckenberger and Kolb in \cite{HKfodcQHS}, every left $\OO_q(G)$-covariant fodc $(\Omega^1,B)$, for which $\Omega^1$ is finitely-generated as a left $B$-module, arises in this way. This gives a bijective correspondence between these two structures.

\begin{remark}
Tangent spaces over quantum homogeneous spaces can be given a more formal treatment in terms of Takeuchi's categorical equivalence. See \cite{HKfodcQHS} for more details.
\end{remark}

\subsection{Quantum Root Vectors}

Let $w_0 = w_{i_1}\cdots w_{i_k}$ be a reduced decomposition of the longest element of the Weyl group of $\mathfrak{g}$. It can be shown that the sequence
\begin{align*}
\beta_1 := \alpha_{i_1}, ~ \beta_2 := w_{i_1}(\alpha_{i_2}), ~ \dots, ~ \beta_k := w_{i_1}\cdots w_{i_k-1}(\alpha_{i_k})
\end{align*}
exhausts all the positive roots of $\mathfrak{g}$. This motivates the following definition of root vector in the noncommutative setting, as presented in detail in \cite{LusztigLeabh}, or \cite[\textsection 6.2]{KSLeabh}.

For $q \in \mathbb{C}\backslash \{0\}$, the {\em quantum integer} $[m]_q$ is the complex number 
$$
[m]_q := q^{-m+1} + q^{-m+3} + \cdots + q^{m-3} + q^{m-1}.
$$
Moreover, we have the \emph{quantum factorials}
\begin{align*}
[n]_q! = [n]_q[n-1]_q \cdots [2]_q[1]_q, & & [0]_q! = 1.
\end{align*}
To every $i=1, \ldots, l$, there corresponds an algebra automorphism $T_i$ of $U_q(\mathfrak{g})$ which acts on the generators as
\begin{flalign*}
T_i(K_j) = K_j K_i^{-a_{ij}}, ~~~~~ T_i(E_i) = -F_iK_i, ~~~~~ T_i(F_i)= -K_i^{-1}E_i, \\
T_i(E_j) = \sum_{t=0}^{-a_{ij}} (-1)^{t-a_{ij}}q_i^{-t}(E_i)^{(-a_{ij}-t)}E_j(E_i)^{(t)}, ~~~ \textrm{ for } i\neq j,~~~~\,\\
T_i(F_j) = \sum_{t=0}^{-a_{ij}} (-1)^{t-a_{ij}}q_i^{t}(F_i)^{(t)}E_j(F_i)^{(-a_{ij}-t)}, ~~~~ \textrm{ for } i\neq j, ~~~~~\,\,
\end{flalign*}
where $(E_i)^{(n)} := E_i^n /[n]_q!$, and $(F_i)^{(n)} := F_i^n / [n]_q!$. The mapping $w_i \rightarrow T_i$ determines a homomorphism of the braid group $\mathfrak{B}_{\mathfrak{g}}$ into the group of algebra automorphisms of $U_q(\mathfrak{g})$. Associated to a choice of reduced decomposition of $w_0 = w_{i_1}\cdots w_{i_k}$ of the longest element of the Weyl group, we have the elements
\begin{align*}
E_{\beta_r} : = T_{i_1}T_{i_2} \cdots T_{i_{r-1}}(E_{i_r}), & & \textrm{ for } r = 1, \dots, k.
\end{align*}
We call any such element a \emph{root vector} of $U_q(\mathfrak{g})$. Just as in the classical case, the root vectors of $U_q(\mathfrak{g})$ have an associated PBW basis \cite[\textsection 6.2.3]{KSLeabh}.

\subsection{The Heckenberger--Kolb $(0,1)$-Calculus} \label{subsection:HKTangent}

We now present the Heckenberger-Kolb calculus, or more precisely the $(0,1)$-summand of the Heckenberger--Kolb calculus (the $(1,0)$-summand will be presented in \textsection \ref{subsection:OppBW}). 

We begin by introducing some notation: let $R \subseteq \mathfrak{h}^*$ denote the root system associated to our choice of Cartan sub-algebra $\mathfrak{h}$ and let $R^+$, respectively $R^-$, be the set of positive, respectively negative, roots. Denote \begin{align*}
  R_S^{\pm}:= \mathbb{Z}S \cap R^{\pm}, & & \overline{R^{\pm}_S} := R^{\pm} \setminus R_S^{\pm}.
\end{align*}

\begin{defn} The \emph{Heckenberger--Kolb calculus} is the left $\OO_q(G)$-covariant fodc calculus associated to the tangent space 
$$
T^{(0,1)} := \mathrm{span}_{\mathbb{C}}\!\left\{E_\beta \,|\, \beta \in \overline{R^+_S} \right\} \subseteq \OO_q(G/L_S)^{\circ},
$$
where by abuse of notation, $E_\beta$ denotes the coset of the root vector  in $\OO_q(G/L_S)^{\circ}$.
\end{defn}

Here we have chosen to describe the calculus explicitly in terms of its corresponding tangent space. Alternatively the calculus can be described as one of the two left $\OO_q(G)$-covariant finite-dimensional irreducible fodc over $\OO_q(G/L_S)$, see \cite{HK}, or \cite{HVBQFM} for a presentation in the notation of the present paper. 

\subsection{Liouville's Theorem}

Since the classical homogeneous space $G/L_S$ is a compact complex manifold, it follows from Liouville's theorem that its only globally holomorphic functions are scalar multiples of the identity. In this subsection we show that this result extends to the quantum setting using the tangent space description of the $(0,1)$-Heckenberger--Kolb calculus. This result will be used in \textsection \ref{section:BorelWeil} when we prove the Borel--Weil theorem for negative line bundles $\EE_{-k}$.

\begin{thm} \label{thm:Liouville}
For any irreducible quantum flag manifold $\OO_q(G/L_S)$, with associated $(0,1)$-Heckenberger-Kolb calculus $\Omega_q^{(0,1)}(G/L_S)$, it holds that 
\begin{align*}
H^{0}_q(G/L_S) = \mathrm{ker}\left(\overline{\partial}:\OO_q(G/L_S) \rightarrow \Omega_q^{(0,1)}(G/L_S) \right) = \mathbb{C}1.
\end{align*}
\end{thm}
\begin{proof}
Since $\alpha_x$ cannot be contained in the $\mathbb{Z}$-span of $S$, we must have that $\alpha_s \in \overline{R_S^+}$. Hence we must have that $E_x \in T^{(0,1)}$. Next take an arbitrary $b \in \OO_q(G/L_S)$. It follows from the definition of the exterior derivative given in \eqref{eqn:tangent.exd}, and the fact that $E_x \in T^{(0,1)}$, that $\overline{\partial}(b)=0$ only if 
$
E_x \triangleright b = 0.
$
However, since $\OO_q(G/L_S)$ is by definition the space of invariants of $\uslevi$, we know that $E_s \triangleright b = 0$, for all $s \neq x$, and moreover, that $b$ is a weight vector of weight zero. Taken together, these two facts imply that $b$ is a highest weight vector of degree zero. However, it follows from the Peter--Weyl decomposition of $\OO_q(G)$ that the only such elements are of the form $\lambda 1$, for some $\lambda \in \mathbb{C}$.
\end{proof}

\begin{remark}
As established in \cite{MarcoConj}, the Heckenberger--Kolb calculi all 
possess noncommutative K\"ahler structures in the sense of \cite[Definition 7.1]{MMF3}. Since the Dolbeault cohomology of any differential calculus endowed with a K\"ahler structure refines its de Rham cohomology \cite[Corollary 7.7]{MMF3}, we have that the three cohomology groups $H^0_{\exd}$, $H^0_{\del}$, and $H^0_{\adel}$ all coincide. This motivates our decision to denote the cohomology in Theorem \ref{thm:Liouville} as $H^0_q(G/L_S)$, without any choice of reference to the exterior derivative involved.
\end{remark}


\section{A Quantum Principal Bundle Presentation of $\Omega^{(0,1)}_q(G/L_S)$}

In this subsection we recall Heckenberger and Kolb's construction of a left $\OO_q(G)$-covariant fodc over the Poisson quantum homogeneous space $\sslevi$. We then show that this calculus is right $\OO(U_1)$-covariant and hence gives a quantum principal bundle presentation of the anti-holomorphic Heckenberger-Kolb calculi on the irreducible quantum flag manifolds.

\subsection{An Alternative Construction of $\sslevi$} \label{subsection:AltConst}

We now recall the alternative description of $\OO_q(G/L^{\,\mathrm{s}}_S)$ given by Heckenberger and Kolb in \cite[\textsection 3.1]{HKdR}. This alternative description is needed for the construction of Heckenberger and Kolb's fodc over $\sslevi$ in the following subsection. We begin by introducing the quantum homogeneous coordinate rings.

\begin{defn} \label{defn:homoCR}
For any irreducible quantum flag manifold, the \emph{homogeneous coordinate ring} $S_q[G/L_S]$ is the subalgebra of $\OO_q(G/L_S)$ generated by the elements $z_i$. Moreover, the \emph{opposite homogeneous coordinate ring} $S_q[G/L^{\mathrm{op}}_S]$ is the subalgebra of $\OO_q(G)$ generated by the elements $\overline{z}_i$.
\end{defn}

It is well known \cite{BraveFlags} that $S_q[G/L_S]$ is a quadratic algebra. Explicitly
$$
S_q[G/L_S] \simeq \mathbb{C}\langle z_1, \dots z_{N_x}\rangle/\langle \sum_{k,l = 1}^{N_x} (R_{V_{\varpi_x},V_{\varpi_x}})^{ij}_{kl} z_k z_l - q^{(\lambda,\lambda)} z_iz_j\rangle.
$$
Similarly, the opposite homogeneous coordinate ring $S_q[G/L^{\mathrm{op}}_S]$ is a quadratic algebra, admitting an analogous $R$-matrix presentation. 
Consider the vector space 
$$
S_q[G/L_S]_{\mathbb{C}} := S_q[G/L_S] \otimes S_q[G/L^{\mathrm{op}}_S].
$$
Endow $S_q[G/L_S]_{\mathbb{C}}$ with the multiplication 
\begin{align}\label{eqn:multiplication}
(1 \otimes \overline{z}_i)(z_j \otimes 1) := q^{(\lambda,\lambda)} \sum_{k,l = 1}^{N_x} (R^{-1}_{V_{\varpi_x},V_{-w_0(\varpi_x)}})^{ij}_{kl} \, z_k \otimes \overline{z}_l.
\end{align}
It follows that 
$$
c:= \sum_{i = 1}^{N_x} (1 \otimes \overline{z}_i)(z_i \otimes 1)
$$
is a central right $U_q(\mathfrak{g})$-invariant element of $S_q[G/L_S]_{\mathbb{C}}$. 
Consider the quotient algebra 
\begin{align*} 
S_q[G/L_S]_{\mathbb{C}}^{c=1} = S_q[G/L_S]_{\mathbb{C}}/\langle c - 1 \rangle,
\end{align*}
where $\langle c- 1 \rangle$ denotes the two-sided ideal of $S_q[G/L_S]_{\mathbb{C}}$ generated by the element $c-1$. It follows from \cite[Lemma 3.1]{HKdR} that the multiplication map 
\begin{align*}
m: S_q[G/L_S]_{\mathbb{C}}^{c=1} \to \OO_q(G/L^{\mathrm{s}}_S), 
\end{align*}
is an isomorphism, giving us the claimed alternative description of $\OO_q(G/L^{\mathrm{s}}_S)$. As an immediate consequence, we get the identity
\begin{align} \label{eqn:GenDet}
\sum_{i = 1}^{N_x} \overline{z}_iz_i = 1.
\end{align}
As quadratic algebras both $S_q[G/L_S]$ and $S_q[G/L^{\mathrm{op}}_S]$ have a natural $\mathbb{Z}_{\geq 0}$-grading. This allows us to define a $\mathbb{Z}$-grading on $S_q[G/L_S]_{\mathbb{C}}$ by setting 
\begin{align} \label{eqn:gradingonSq}
\mathrm{deg}(z_i) = 1, & & \mathrm{deg}(\overline{z}_i) = -1.
\end{align}
Since $c$ is of degree zero, the grading descends to a $\mathbb{Z}$-grading on $S_q[G/L_S]_{\mathbb{C}}^{c=1}$, giving us a $\mathbb{Z}$-grading on $\OO_q(G/L^{\mathrm{s}}_S)$.

\subsection{The Heckenberger--Kolb First-Order Differential Calculus on $\OO_q(G/L^{\mathrm{s}}_S)$}

Define $\Gamma$ to be the left $S_q[G/L_S]$-module generated by the elements $\adel z_i$, for $i = 1, \dots, N_x$, subject to the relations
$$
\sum_{k,l = 1}^{N_x} \left[R^2_{V_{\varpi_x},V_{\varpi_x}} + q^{(\varpi_x,\varpi_x)}(q^{(\alpha_x,\alpha_x)}-1)R_{V_{\varpi_x},V_{\varpi_x}} + q^{2(\varpi_x,\varpi_x) - (\alpha_x,\alpha_x)}\id\right]^{kl}_{ij} z_i \adel z_j = 0,
$$
for all $i,j = 1, \dots, N_x$. As shown in \cite[\textsection 3.2.3]{HKdR}, we can endow $\Gamma$ with a left $S_q[G/L_S]_{\mathbb{C}}$-module structure such that 
a derivation 
$
\adel:S_q[G/L_S]_{\mathbb{C}} \to S_q[G/L_S] \otimes \Gamma,
$
is uniquely determined by 
$$
\adel:f \otimes \overline{z}_i \mapsto f \otimes \adel \overline{z}_i. 
$$
This gives us a left $\OO_q(G)$-covariant fodc over $S_q[G/L_S]_{\mathbb{C}}$.

Consider the sub-bimodule of $S_q[G/L_S] \otimes \Gamma$ generated by the form $\adel c$ and the subsets 
\begin{align*}
  (c-1)(S_q[G/L_S] \otimes \Gamma), & & (S_q[G/L_S] \otimes \Gamma)(c-1).
\end{align*}
Quotienting $S_q[G/L_S] \otimes \Gamma$ by this sub-bimodule, we get a left $\OO_q(G)$-covariant fodc over $\OO_q(G/L^{\,\mathrm{s}}_S)$, which we denote by
$$
\left(\,\Omega^{(0,1)}_q(G/L^{\,\mathrm{s}}_S), \, \adel\,\right)\!.
$$
We see that by construction $\adel f = 0$, for all $f \in S_q[G/L_S]$.


\subsection{Principal Pairs and Quantum Principal Bundles}

We now recall the notion of a principal pair, introduced in \cite{KMOS} as a framework in which to construct quantum principal bundles. We use the equivalent, yet simpler, formulation introduced in \mbox{\cite[Definition 3.8]{GAPP}}.

\begin{defn}
For a Hopf algebra $A$, a \emph{principal $A$-pair} is a pair of quantum homogeneous $A$-spaces $B \subseteq P$, such that  
$
\pi_B(P) = \pi_B(A)^{\co(\pi_P(A))}
$
is a Hopf subalgebra of $\pi_B(A)$, with respect to the coaction 
\begin{align*}
    \pi_B(P) \to \pi_B(P) \otimes \pi_P(A), & & \pi_B(p) \mapsto \pi_B(p_{(1)}) \otimes \pi_P(p_{(2)}).
\end{align*}
\end{defn}

For principal pairs, the quantum principal bundle condition reduces to a much simpler covariance requirement on the total space calculus, as we now recall.

\begin{prop} \label{prop:PPQPB}
For a principal pair $B \subseteq P$, let $\Omega^1(P)$ be a left $A$-covariant, and right $\pi_B(A)$-covariant, fodc over $P$. It holds that the pair
$
(P,\Omega^1(P))
$
is a quantum principal bundle.
\end{prop}

We will use this result in the following subsection to produce a quantum principal bundle presentation of the Heckenberger--Kolb calculi.

\subsection{A Quantum Principal Bundle} 

The motivating example of a principal pair was the Poisson circle bundle over the quantum Grassmannians. In \cite{GAPP} this was extended to include all Poisson torus bundles over quantum flag manifolds, as we now recall. 

\begin{thm} For any Drinfeld--Jimbo quantum group $U_q(\mathfrak{g})$, and any subset $S \subseteq \Pi$ of simple roots of $\mathfrak{g}$, the pair
$$
\left(\OO_q(G/L_S), \, \OO_q(G/L^{\,\mathrm{s}}_S)\right)
$$
is a principal pair of quantum homogeneous $\OO_q(G)$-spaces.
\end{thm}

This means that we can now use Proposition \ref{prop:PPQPB} to produce a quantum principal bundle from the Heckenberger--Kolb calculi. We first establish right $\OO(U_1)$-covariance of the calculus over $\sslevi$, and then conclude the quantum principle structure.

\begin{lem}
The fodc $\Omega^{(0,1)}_q(G/L^{\,\mathrm{s}}_S)$ is right $\OO(U_1)$-covariant.
\end{lem}
\begin{proof}
Consider the $\mathbb{Z}$-grading on $\Gamma$ defined by 
\begin{align*}
\deg(z_i) = \deg(\adel z_i) = 1.
\end{align*}
Since the defining relations of $\Gamma$ are of degree $2$ with respect to this grading, it must descend to a well-defined $\mathbb{Z}$-grading on $\Gamma$. Moreover, since $\adel$ is clearly a degree zero map, we have that the fodc is right $\OO(U_1)$-covariant.

With respect to the grading on $S_q[G/L^{\mathrm{op}}_S]$, inherited from the grading of $\sslevi$, we take the tensor product grading on 
\begin{align}\label{eqn:gammaC}
S[G/L^{\mathrm{op}}_S] \otimes \Gamma.
\end{align}
With respect to this grading, the differential $\adel$ is again a degree zero map, meaning that we have a right $\OO(U_1)$-covariant fodc.

Finally, we note that the elements $c-1$ and $\overline{\partial} c$ are homogeneous and of degree zero. This means that the sub-bimodule they generate is homogeneous, and so, the $\mathbb{Z}$-grading of \eqref{eqn:gammaC} descends to a $\mathbb{Z}$-grading on $\Omega^{(0,1)}_q(G/L^{\,\mathrm{s}}_S)$, giving us a right $\OO(U_1)$-covariant fodc.
\end{proof}

This result, taken together with Proposition \ref{prop:PPQPB} implies the following proposition.

\begin{prop}
The pair 
$$
\left(\Omega^{(0,1)}_q(G/L_S^{\,\mathrm{s}}),\Delta_{R, \, \OO(U_1)}\right)\!
$$
is a quantum principal bundle.
\end{prop}


\subsection{A Left $\OO_q(G)$-Covariant Principal Connection}

In this subsection we construct a strong principal connection for our bundle. We begin with some results about the horizontal forms of the bundle.

\begin{lem}
For the quantum principal bundle $(\Omega^{(0,1)}_q(G/L_S^{\mathrm{s}}),\Delta_R)$, it holds that 
\begin{enumerate}
  \item $\Omega^{(0,1)}_q(G/L^{\mathrm{s}}_S)_{\mathrm{hor}} = \Omega^{(0,1)}_q(G/L_S)\qphs$,
  \item $\Omega^{(0,1)}_q(G/L^{\mathrm{s}}_S)_{\mathrm{ver}} = 0$.
\end{enumerate}
\end{lem}
\begin{proof}
1. ~An arbitrary form in $\Omega^{(0,1)}_q(G/L_S)$ is a sum of elements of the form $b\,\adel b'$, for $b,b' \in \OO_q(G/L_S)$. Consider now the form 
\begin{align*}
z_i b\,\adel b' \in \qphs \Omega^{(0,1)}_q(G/L_S), & & \textrm{ for any } i = 1, \dots, N_x. 
\end{align*}
It follows from \eqref{eqn:multiplication} and \eqref{eqn:GenDet} that there exist elements $f_j\in S_q[G/L_S]$ and $v_j \in S_q[G/L^{\mathrm{op}}_S]$ such that $\sum_j f_jv_j=1.$ Then it follows from the Leibniz rule that 
\begin{align*}
z_i b\adel b' = \adel(z_ibb') - \adel(z_ib)b' = & \, \sum_j \adel(z_ibb')f_jv_j - \sum_j \adel(z_i b)f_jv_jb'.
\end{align*}
Moreover, since $\adel f_j = 0$, we have that 
\begin{align*}
 \sum_j \adel(z_ibb')f_jv_j - \sum_j \adel(z_i b)f_jv_jb' = \sum_j \adel(z_ibb'f_j)v_j - \sum_j \adel(z_ibf_j)v_jb'.
\end{align*}
Thus we see that $z_i b\adel b'$ is an element of $\Omega^{(0,1)}_q(G/L_S)\qphs$, for any $i = 1, \dots, N_x$. An analogous argument establishes that 
$$
\overline{z}_i b\adel b' \in \Omega^{(0,1)}_q(G/L_S)\qphs.
$$
%
Thus we can conclude that 
\begin{align*}
\qphs \Omega^{(0,1)}_q(G/L_S) \subseteq \Omega^{(0,1)}_q(G/L_S)\qphs.
\end{align*}

2. ~Consider now a general element $\omega \in \Omega^{(0,1)}_q(G/L_S^{\,\mathrm{s}})$. It follows from the Leibniz rule, and the fact that $\adel z_i = 0$, for all $i =1, \dots, N_x$, that $\omega$ is a sum of the form 
\begin{align*}
\sum_{i} p_i (\adel \overline{z}_i)p'_i, & & \textrm{ for } p_i,p'_i \in \OO_q(G/L_S^{\,\mathrm{s}}).
\end{align*}
Using the same argument as above, we see that
\begin{align*}
p_i (\adel \overline{z}_i)p'_i = & \, \sum_j p_i (\adel \overline{z}_i)f_jv_jp_i'
= \sum_j p_i \adel(\overline{z}_if_j)v_jp_i'.
\end{align*}
It now follows from the presentation of the horizontal forms given in 1. that  
\begin{align*}
\Omega^{(0,1)}_q(G/L_S^{\,\mathrm{s}}) = \Omega^{(0,1)}_q(G/L_S) \qphs.
\end{align*}
Thus we see that $\Omega^{(0,1)}_q(G/L^{\mathrm{s}}_S)_{\mathrm{ver}} = 0$.
\end{proof}

From this lemma, and the definition of a strong principal connection, we get the following proposition.

\begin{prop}
The zero map on $\Omega^{(0,1)}_q(G/L^{\mathrm{s}}_S)_{\mathrm{hor}}$ is a left $\OO_q(G)$-covariant strong principal connection.
\end{prop}

\begin{remark}
The zero projection can in fact be realised as the restriction of any universal connection, just as for the quantum Grassmannians case \cite{KMOS}. For example, an explicit construction of a universal connection is given in \cite[Proposition 5.9]{GAPP}.
\end{remark}

\subsection{The Holomorphic Structure of the Relative Line Modules}\label{subsection:HoloStructures}

We denote by 
$$
\OO_q(G/L^{\mathrm{s}}_S) \simeq \bigoplus_{k \in \mathbb{Z}} \EE_k
$$
the decomposition of $\OO_q(G/L^{\mathrm{s}}_S)$ into homogeneous subspaces given by the $\mathbb{Z}$-grading defined in \textsection \ref{subsection:AltConst}. It follows from \cite[Proposition 5.7]{GAPP} that this is a decomposition of $\OO_q(G/L^{\mathrm{s}}_S)$ into covariant line modules, and moreover, that every relative line module is of this form.


The principal connection identified in the subsection above associates to each $\EE_k$ a left $\OO_q(G)$-covariant connection, which by uniqueness must coincide with the holomorphic structure $\adel_{\EE_k}$. Thus we have produced a principal connection presentation of the holomorphic structure of $\EE_k$. Explicitly, the holomorphic structure can be described as
\begin{align} \label{eqn:explicitHS}
\adel_{\EE_k}: \EE_k \to \Omega^{(0,1)} \otimes_{\OO_q(G/L_S)} \EE_k, & & e \mapsto j^{-1} \circ \adel(e),
\end{align}
where $j$ is the map defined in \textsection \ref{subsection:ConnFromPrinCs}.

\section{A Borel--Weil Theorem for the Irreducible Quantum Flag Manifolds} \label{section:BorelWeil}

In this section we prove the main result of the paper, namely the Borel-Weil theorem for $(0,1)$-Heckenberger--Kolb calculi of the irreducible quantum flag manifolds $\OO_q(G/L_S)$. We also discuss the situation for the opposite complex structure of $\OO_q(G/L_S)$, and give a noncommutative geometric presentation of the quantum homogeneous ring of $S_q[G/L_S]$.

\subsection{The Borel--Weil Theorem} \label{subsection:BW}

In this subsection we prove a direct noncommutative generalisation of the Borel-Weil theorem for the irreducible flag manifolds. 

\begin{thm} \label{thm:BW}
For any irreducible quantum flag manifold $\OO_q(G/L_S)$, it holds that 
\begin{enumerate}
 \item $H^0(\mathcal{E}_k) = z^kU_q(\mathfrak{g})$ is an irreducible $U_q(\mathfrak{g})$-module of highest weight $k\varpi_{x}$,
 \item $H^0(\mathcal{E}_{-k}) = 0$.
\end{enumerate}
for all $k \in \mathbb{Z}_{>0}$, where $\Pi\backslash S$ consists of the simple root $\alpha_x$.
\end{thm}
\begin{proof}
~~~ \\
1. ~From the construction of the Heckenberger--Kolb calculus $\Omega^{(0,1)}_q(G/L^{\,\mathrm{s}}_S)$, we know that $z^k$ is contained in the kernel of the differential $\adel$, for all $k \in \mathbb{Z}_{\geq 0}$. Thus it follows from \eqref{eqn:explicitHS} that $z^k$ is contained in the kernel of the holomorphic structure map $\adel_{\EE_k}$. Since $\adel_{\EE_k}$ is a right $U_q(\mathfrak{g})$-module map, the inclusion 
$$
z^k U_q(\mathfrak{g}) \subseteq H^0(\mathcal{E}_k)
$$
now follows from Schur's lemma.

To show the opposite inclusion, we need to consider the decomposition of $\EE_k$ into irreducible right $U_q(\mathfrak{g})$-modules. As established in Proposition \ref{prop:bzHW}, each irreducible subcomodule contains an element of the form $bz^k$, for some $b \in \OO_q(G/L_S)$. Let us assume that one of these elements $bz^k$, for $b \notin \mathbb{C}1$, were holomorphic. 
This would imply that
  \begin{align*}
 0 =  \adel_{\EE_k}\!(b z^k) = \adel b \otimes z^{k} + b \otimes \adel_{\EE_k}(z^k)
     = \adel b \otimes z^k.
   \end{align*}
Noting that $\Omega^{(0,1)}_q\!\left(G/L^{\mathrm{\,s}}_S\right)$ is a torsion-free right $\OO_q(G/L_S)$-module, that $\mathcal{E}_k$ is projective as left $\OO_q(G/L_S)$-module and that $\OO_q(G/L_S)$ has no zero-divisors, we see that Liouville's theorem implies that 
$$
\adel b \otimes z^k \neq 0.
$$
Thus we see that no such holomorphic element exists, and so, the claimed identity $z^k U_q(\mathfrak{g}) \subseteq H^0(\mathcal{E}_k)$ follows from Schur's lemma. Moreover, it is clear that $z^k U_q(\mathfrak{g})$ is an irreducible $U_q(\mathfrak{g})$-module of highest weight $k\varpi_s$.

\bigskip

2. ~ We now come to the line bundles $\mathcal{E}_{-k}$, for $k \in \mathbb{Z}_{>0}$. Assume there exists a non-zero holomorphic element $e \in \mathcal{E}_{-k}$. 
This implies that, for any $i=1, \dots, N_x$, we have
\begin{align*}
~~~~ \adel\big(ez^{\,k}_i\big) = & \, \adel\big(ez^{\,k}_i\big) 
= \, \adel(e)z^{\,k}_i + e\adel(z^{\,k}_i)
= 0,
\end{align*}
which is to say 
$$
ez^{\,k}_i \in H^0_{q}(G/L_S) := \ker(\adel:\OO_q(G/L_S) \to \Omega^{(0,1)}_q(G/L_S)).
$$ 
Liouville's Theorem \ref{thm:Liouville} implies that $e z_i^k$ must be a non-zero scalar multiple of $1$, and so, we must have a distinct right inverse to $e$, for each $i = 1, \dots, N_x$. However, since $\OO_q(G)$ has no zero divisors, right inverses are unique. To avoid contradiction we are forced to conclude that $\EE_{-k}$ contains no holomorphic elements. 
\end{proof}

\subsection{A Holomorphic Description of the Quantum Homogeneous Coordinate Ring}

Any positive line bundle $\EE$ over a compact K\"ahler manifold $M$ gives an embedding of $M$ into complex projective space $\mathbb{CP}^n$. Moreover, the associated homogeneous coordinate ring $S(M)$ is isomorphic to
\begin{align} \label{eqn:HoloHCR}
\bigoplus_{k \in \mathbb{Z}_{\geq 0}} H^0(\EE^{\otimes k}),
\end{align}
as a graded algebra. For the flag manifolds, the line bundle $\EE_1$ is well-known to be positive, and hence we have an associated projective embedding. For the special case of the Grassmannians, this embedding reduces to the celebrated Pl\"ucker embedding.

The following proposition establishes a quantum generalisation of this result, giving a noncommutative differential geometric realisation of the quantum homogeneous coordinate ring $S_q[G/L_S]$ presented in the Definition \ref{defn:homoCR}. This generalises earlier work in \cite{KLVSCP1,KKCP2,KKCPN} for the case of quantum projective space, and work by Mrozinski and the first and third authors in \cite{KMOS} for the more general family of the quantum Grassmannians.

\begin{prop} \label{prop:HHCR}
For any irreducible quantum flag manifold $\OO_q(G/L_S)$, it holds that 
\begin{align} \label{eqn:HCR}
\bigoplus_{k \in \mathbb{Z}_{\geq 0}} H^0(\mathcal{E}_k) = S_q[G/L_S],
\end{align}
where $\EE_{\pm k}$ is a line module over $\OO_q(G/L_S)$, with $k \in \mathbb{Z}_{\geq 0}$, and both algebras are considered as subalgebras of $\OO_q(G)$.
\end{prop}
\begin{proof}
Recall from Theorem \ref{thm:BW} that $H^0(\EE_k) = z^k U_q(\mathfrak{g})$. Thus since $S_q[G/L_S]$ is a right $U_q(\mathfrak{g})$-submodule of $\OO_q(G)$, it follows that 
\begin{align*}
\bigoplus_{k \in \mathbb{Z}_{\geq 0}} H^0(\EE_k) \sseq S_q[G/L_S].
\end{align*}
We now establish the opposite inclusion. For $e \in H^0(\EE_k), \, e' \in H^0(\EE_l)$, observe that 
\begin{align*}
 \adel_{\EE_{k+l}}(ee') = & \, j^{-1}\!\left(\adel(ee')\right) 
 = \, j^{-1}\!\left((\adel e)e' + e\adel(e')\right) 
 = \, j^{-1}\left(\adel e)e' + e\adel e'\right) 
 = \, 0.
\end{align*}
Thus the right hand side of the equality \eqref{eqn:HCR} is a subalgebra of $\OO_q(G)$. Now $z_i \in z U_q(\mathfrak{g})$, for each $i=1, \dots, N_x$, which is to say, the generators of $S_q[G/L_S]$ are contained in $H^0(\mathcal{E}_1)$. Thus we have the opposite inclusion
$$
S_q[G/L_S] \sseq \bigoplus_{k \in \mathbb{Z}_0} H^0(\mathcal{E}_k),
$$ 
and hence equality of the two algebras.
\end{proof} 

\begin{remark} \label{rem:positiveEE1}
This differential geometric description of the quantum homogeneous coordinate ring of the irreducible quantum flag manifolds hints at the existence of a noncommutative generalisation of the classical \emph{g\'eom\'etrie alg\'ebrique et g\'eometrie analytique} correspondence \cite{SerreGAGA}. See \cite[\textsection 7]{BS} for a detailed discussion of how such a general picture might look.
\end{remark}


\subsection{The Opposite Borel--Weil Theorem} \label{subsection:OppBW}

As discussed in \textsection \ref{subsection:HKTangent}, the Heckenberger--Kolb $(0,1)$-fodc is one of two calculi identified in the classification of left $\OO_q(G)$-covariant finite-dimensional irreducible calculi over the irreducible quantum flag manifolds. We call this other fodc the \emph{Heckenberger--Kolb $(1,0)$-calculus} and denote it by $\Omega^{(1,0)}_q(G/L_S)$. Its tangent space $T^{(1,0)}$ is given by 
\begin{align*}
T^{(1,0)} = \mathrm{span}\left\{F_{\beta} \,|\, \beta \in \overline{R^-_S} \right\}.
\end{align*}
Just as for the $(0,1)$-calculus, each line bundle $\EE_k$ over $\OO_q(G/L_S)$ admits a unique left $\OO_q(G)$-covariant connection 
$$
\del_{\EE_k}: \EE_k \to \Omega^{(0,1)} \otimes_{\OO_q(G/L_S)} \EE_k.
$$
Using a construction dual to that of $(\Omega^{(0,1)}_q(G/L^{\,\mathrm{s}}_S),\adel)$, a covariant fodc $(\Omega^{(1,0)}_q(G/L^{\,\mathrm{s}}_S),\del)$ was introduced in  \cite[3.2.3]{HKdR}. This fodc restricts to the Heckenberger--Kolb $(1,0)$-calculus on $\OO_q(G/L_S)$, and satisfies $\del v = 0$, for all $v \in S_q[G/L_S^{\mathrm{op}}]$. Just as for the $(0,1)$-case, this fodc can be used to give quantum principal bundle presentation of the Heckenberger--Kolb $(1,0)$-fodc, as well as a principal connection description of the connections $\adel_{\EE_k}$, leading to the following $(1,0)$-version of the Borel--Weil theorem. 

\begin{thm} 
For any irreducible quantum flag manifold $\OO_q(G/L_S)$, it holds that
\begin{enumerate}
 \item $H^0(\mathcal{E}_{k}) = 0$,
 \item $H^0(\mathcal{E}_{-k}) = \overline{z}^kU_q(\mathfrak{g})$ is an irreducible $U_q(\mathfrak{g})$-module of highest weight $-w_0(k\varpi_{x})$,
\end{enumerate}
for all $k \in \mathbb{Z}_{>0}$, where $\Pi\backslash S$ consists of the simple root $\alpha_x$, and $H^0 := \ker(\del)$.
\end{thm}

Just as for the $(0,1)$-case, we can use this result to produce a noncommutative differential geometric description of the opposite quantum homogeneous coordinate ring. 

\begin{cor} 
It holds that
\begin{align} 
\bigoplus_{k \in \mathbb{Z}_{\geq 0}} H^0(\mathcal{E}_{-k}) = S_q[G/L^{\mathrm{op}}_S],
\end{align}
where both algebras are considered as subalgebras of $\OO_q(G)$.
\end{cor}

\appendix

\section{Relative Line Modules and Spherical Generators}

To prove the Borel--Weil theorem we used the fact that every highest weight vector in $\EE_k$ is for the form $bz^k$, for some $b \in \OO_q(G/L_S)$. The proof is essentially the same as in the classical case, but for the reader's convenience, we give an explicit proof.

\subsection{The Table of Spherical Weights}

In Table \ref{table:CQFMs} below we recall the precise subsets $S$ of simple roots $\Pi$ defining the irreducible quantum flag manifolds. We represent $S^c := \Pi\backslash S$ graphically by coloured nodes in the Dynkin diagram of $\mathfrak{g}$. 

The set of highest weights of the $U_q(\mathfrak{g})$-algebra $\mathcal{O}_q(G/L_S)$ forms an additive submonoid $Z_S \subseteq \mathfrak{h}^*$ under addition. A distinguished minimal set of generators for $Z_S$, the {\em spherical weights}, was presented in the classical case by Kr\"amer in \cite[Tabelle 1]{KramerSpherical}. 

\begin{center}
  \captionof{table}{\small{Irreducible quantum flag manifolds with defining crossed node} \\
  ~~ } \label{table:CQFMs}
 \begin{tabular}{|c|c|c|}

\hline
& & \\
$A_n$&
\begin{tikzpicture}[scale=.5]
\draw
(0,0) circle [radius=.25] 
(8,0) circle [radius=.25] 
(2,0) circle [radius=.25] 
(6,0) circle [radius=.25] ; 

\draw[fill=black]
(4,0) circle [radius=.25] ;

\draw[thick,dotted]
(2.25,0) -- (3.75,0)
(4.25,0) -- (5.75,0);

\draw[thick]
(.25,0) -- (1.75,0)
(6.25,0) -- (7.75,0);
\end{tikzpicture}  & $\OO_q(\mathrm{Gr}_{n,m})$\\

& & \\

$B_n$ &
\begin{tikzpicture}[scale=.5]
\draw
(4,0) circle [radius=.25] 
(2,0) circle [radius=.25] 
(6,0) circle [radius=.25] 
(8,0) circle [radius=.25] ; 
\draw[fill=black]
(0,0) circle [radius=.25];

\draw[thick]
(.25,0) -- (1.75,0);

\draw[thick,dotted]
(2.25,0) -- (3.75,0)
(4.25,0) -- (5.75,0);

\draw[thick] 
(6.25,-.06) --++ (1.5,0)
(6.25,+.06) --++ (1.5,0);           

\draw[thick]
(7,0.15) --++ (-60:.2)
(7,-.15) --++ (60:.2);
\end{tikzpicture} & $\OO_q(\mathbf{Q}_{2n+1})$  \\ 

& & \\

 $C_n$& 
\begin{tikzpicture}[scale=.5]
\draw
(0,0) circle [radius=.25] 
(2,0) circle [radius=.25] 
(4,0) circle [radius=.25] 
(6,0) circle [radius=.25] ; 
\draw[fill=black]
(8,0) circle [radius=.25];

\draw[thick]
(.25,0) -- (1.75,0);

\draw[thick,dotted]
(2.25,0) -- (3.75,0)
(4.25,0) -- (5.75,0);

\draw[thick] 
(6.25,-.06) --++ (1.5,0)
(6.25,+.06) --++ (1.5,0);           

\draw[thick]
(7,0) --++ (60:.2)
(7,0) --++ (-60:.2);
\end{tikzpicture} &  $\OO_q(\mathbf{L}_{n})$   \\ 

& & \\

 $D_n$& 
\begin{tikzpicture}[scale=.5]

\draw[fill=black]
(0,0) circle [radius=.25] ;

\draw
(2,0) circle [radius=.25] 
(4,0) circle [radius=.25] 
(6,.5) circle [radius=.25] 
(6,-.5) circle [radius=.25];

\draw[thick]
(.25,0) -- (1.75,0)
(4.25,0.1) -- (5.75,.5)
(4.25,-0.1) -- (5.75,-.5);

\draw[thick,dotted]
(2.25,0) -- (3.75,0);
\end{tikzpicture} &  $\OO_q(\mathbf{Q}_{2n})$  \\ 

& & \\

 $D_n$ & 
\begin{tikzpicture}[scale=.5]
\draw
(0,0) circle [radius=.25] 
(2,0) circle [radius=.25] 
(4,0) circle [radius=.25] ;

\draw[fill=black] 
(6,.5) circle [radius=.25];
\draw
(6,-.5) circle [radius=.25];

\draw[thick]
(.25,0) -- (1.75,0)
(4.25,0.1) -- (5.75,.5)
(4.25,-0.1) -- (5.75,-.5);

\draw[thick,dotted]
(2.25,0) -- (3.75,0);
\end{tikzpicture} &  $\OO_q(\textbf{S}_{n})$  \\

 $E_6$& \begin{tikzpicture}[scale=.5]
\draw
(2,0) circle [radius=.25] 
(4,0) circle [radius=.25] 
(4,1) circle [radius=.25]
(6,0) circle [radius=.25] ;

\draw
(0,0) circle [radius=.25];
\draw[fill=black] 
(8,0) circle [radius=.25];

\draw[thick]
(.25,0) -- (1.75,0)
(2.25,0) -- (3.75,0)
(4.25,0) -- (5.75,0)
(6.25,0) -- (7.75,0)
(4,.25) -- (4, .75);
\end{tikzpicture}

 & $\OO_q(\mathbb{OP}^2)$  \\
 
 & & \\
 
 $E_7$& 
\begin{tikzpicture}[scale=.5]
\draw
(0,0) circle [radius=.25] 
(2,0) circle [radius=.25] 
(4,0) circle [radius=.25] 
(4,1) circle [radius=.25]
(6,0) circle [radius=.25] 
(8,0) circle [radius=.25];

\draw[fill=black] 
(10,0) circle [radius=.25];

\draw[thick]
(.25,0) -- (1.75,0)
(2.25,0) -- (3.75,0)
(4.25,0) -- (5.75,0)
(6.25,0) -- (7.75,0)
(8.25, 0) -- (9.75,0)
(4,.25) -- (4, .75);
\end{tikzpicture} &  $\OO_q(\textbf{F})$ 
  \\
 & & \\
\hline
\end{tabular}
\end{center}

\begin{center}
\captionof{table}{\small{ We use Humphrey's numbering of the Dynkin nodes \cite[\textsection 11.4]{Humph} of $\mathfrak{g}$. For the spherical weights of $\O_q(\mathbf{S}_{2m})$, the weight $2\varpi_{2m-1}$ or $2\varpi_{2m}$ appears depending on the defining crossed node, as presented in Table \ref{table:CQFMs} above. \\
  ~~ }} \label{table:SphericalWs}
\begin{tabular}{|c|c|l}
\hline
   & \\
{\em \bf $\O_q(G/L_S)$ }
 & Spherical Weights\\
\hline
   & \\ 
$\O_q(\text{Gr}_{n,m})$  &  $\varpi_{1} + \varpi_{n-1}, \, \varpi_{2} + \varpi_{n-2}, \dots, \varpi_{m} + \varpi_{n-m}$\\
   & \\ 
$\O_q(\mathbf{Q}_{2n+1})$ & $2\varpi_1, \, \varpi_2$ \\    & \\ 
$\O_q(\mathbf{L}_n)$ & $2\varpi_1, \, 2\varpi_2, \dots, 2\varpi_{n}$ \\    & \\ 
$\O_q(\mathbf{Q}_{2n})$ & $2\varpi_1, \, \varpi_2$ \\    & \\ 
$\O_q(\mathbf{S}_{2m})$ & $\varpi_2, \, \varpi_4, \dots, \varpi_{2m-2}, \, 2\varpi_{2m-1} \text{ or } 2\varpi_{2m}$\\    & \\ 
   $\O_q(\mathbf{S}_{2m+1})$ & $\varpi_2, \, \varpi_4, \dots, \varpi_{2m-2}, \, \varpi_{2m} + \varpi_{2m+1}$\\    & \\ 
$\O_q(\mathbb{OP}^2)$ & $\varpi_1 + \varpi_6$,  $\varpi_{2}$\\    & \\ 
$\O_q(\mathbf{F})$ & $\varpi_{1}, \, \varpi_{6}, \, 2\varpi_{7}$\\
 & \\
\hline
\end{tabular}
\end{center}

Since the Weyl character formula is unchanged under $q$-deformation (see \cite[\textsection 7.1.4]{KSLeabh} for details) Kr\"amer's result carries over directly to the quantum setting, and we present it as such in Table \ref{table:SphericalWs} above.

Let us now identify the properties of the spherical weights that are used below. Firstly we highlight the fact that since $G/L_S$ is an Hermitian symmetric space, $\OO_q(G/L_S)$ is multiplicity free as a $U_q(\mathfrak{g})$-module. Secondly we note that the sum of the crossed node $\varpi_x$ and its dual $-w_0(\varpi_x)$ is a spherical generator for each irreducible quantum flag, while neither the crossed node nor its dual appear individually as generators. It is perhaps helpful to observe that in the cases where the crossed node is self-dual (which is to say the spaces $\OO_q(\mathbf{Q}_{2n+1}), \OO_q(\mathbf{L}_{n})$, $\OO_q(\mathbf{Q}_{2n})$, $\OO_q(\mathbf{S}_{2m})$ and $\OO_q(\mathbf{F})$) we have that $\varpi_x -w_0(\varpi_x) = 2\varpi_x$.

\subsection{Spherical Generators for Line bundles}

In this subsection, we produce the required presentation of the highest weight elements of the line bundles $\EE_k$, for $k \in \mathbb{Z}$. The proof relies on the observation that since the product of two highest weight elements of $\OO_q(G)$ is again a highest weight element, the set of highest weight elements of $\OO_q(G)$ forms a submonoid of $A$. (For a more detailed discussion of the properties of highest weight vectors in $U_q(\mathfrak{g})$-algebras see \cite[\textsection 4]{DOS1}.)

\begin{prop} \label{prop:bzHW}
For every $k \in \mathbb{Z}_{\geq 0}$, it holds that 
\begin{enumerate}
\item the set of highest weights of $\EE_k$ is given by
$$
\{\lambda - k w_0(\varpi_x) \,| \, \lambda \textrm{ a highest weight of } \EE_0 \},
$$

\item every highest weight element of $\EE_k$ is of the form $bz^k$, for some $b \in \OO_q(G/L_S)$,

\item the set of highest weights of $\EE_{-k}$ is given by
$$
\{\lambda + k \varpi_x \,| \, \textrm{ for } \lambda \textrm{ is a highest weight of } \EE_0 \},
$$

\item every highest weight element of $\EE_{-k}$ is of the form $b\overline{z}^k$, for some $b \in \OO_q(G/L_S)$.
\end{enumerate}
\end{prop}
\begin{proof}
1,2 ~ Since the element $\overline{z}^k \in \EE_k$ is a highest weight vector, and the highest weight elements of $\OO_q(G)$ are closed under multiplication, we have a well-defined map
\begin{align*}
  M_{\overline{z}^k}: (\EE_k)_{\mathrm{hw}} \to (\EE_0)_{\mathrm{hw}}, & & e \mapsto \overline{z}^ke.
\end{align*}
Moreover, since $\OO_q(G)$ has no zero divisors, $M_{\overline{z}^k}$ is injective. Injectivity implies that the dimensions of the highest weight spaces are preserved. Thus since $\EE_0$ is multiplicity free, $\EE_k$ must be multiplicity free. 

Note next that since $\mathrm{wt}(\overline{z}^k) = - kw_0(\varpi_x)$, we have, for any $e \in (\EE_k)_{\hw}$, that
\begin{align*}
\mathrm{wt}(M_{\overline{z}^k}(e)) = \mathrm{wt}(\overline{z}^ke) = \mathrm{wt}(\overline{z}^k) + \mathrm{wt}(e) = k\varpi_x + \mathrm{wt}(e).
\end{align*}

However, $M_{\overline{z}^k}(e) \in \EE_0$, so there exists an $l \in \mathbb{Z}_{\geq k}$ such that 
$$
k\varpi_x + \mathrm{wt}(e) = l(\varpi_x - w_0(\varpi_x)) + \beta,
$$
where $\beta \in \mathcal{P}^+$ is a weight contained in the $\mathbb{Z}$-span of all the spherical weights apart from $\varpi_x-w_0(w_x)$. This means that 
$$
\mathrm{wt}(e) = (l-k)\varpi_x - lw_0(\varpi_x) + \beta,
$$
or equivalently, that every weight of $\EE_k$ is of the form 
\begin{align}\label{eqn:weightvectors}
\lambda - k w_0(\varpi_x), & & \textrm{ for $\lambda$ a highest weight of $\EE_0$}.
\end{align}

Let us now show that every such weight is realised. Let $b \in \EE_0$ be a highest weight vector of weight $\lambda$. We see that $bz^k_{1} \in \EE_k$, and that 
$$
\mathrm{wt}(z^kb) = - kw_0(\varpi_x) + \lambda.
$$
Thus we have found a weight vector for each weight of the form \eqref{eqn:weightvectors}. Moreover, we have shown that every highest weight vector of $\EE_k$ is of the form $bz^k$.

3,4. ~ The proof for the case of the negative line bundles $\EE_{-k}$ is analogous, and so, we omit it.
\end{proof}


\section{Noncommutative Complex and Holomorphic Structures}

In order to keep the necessarily preliminaries to a minimum, this paper is presented in terms of first-order differential calculi. However, these objects find their fullest geometric presentation when considered as part of a differential graded algebra. This is even more so when one is considered noncommutative complex geometry. In this appendix, with a view to better motivating the results of the paper, we recall the basic definitions and results of the theory of differential calculi, complex structures, and holomorphic structures, as well as a summary of the noncommutative complex geometry of the Heckenberger--Kolb calculi.

\subsection{Differential Calculi}

A {\em differential calculus} $\big(\Omega^\bullet \simeq \bigoplus_{k \in \bZ_{\geq 0}} \Omega^k, \exd\big)$ is a differential graded algebra that is generated as an algebra by the elements $a, \exd b$, for $a,b \in \Omega^0$. 
For a given algebra $B$, a differential calculus {\em over} $B$ is a differential calculus such that $B = \Omega^0$.
A {\em differential $*$-calculus} over a $*$-algebra $B$ is a differential calculus over $B$ such that the \mbox{$*$-map} of $B$ extends to a (necessarily unique) conjugate linear involutive map $*:\Omega^\bullet \to \Omega^\bullet$ satisfying $\exd(\omega^*) = (\exd \omega)^*$, and 
\begin{align*}
\big(\omega \wed \nu\big)^* = (-1)^{kl} \nu^* \wed \omega^*, & & \text{ for all } \omega \in \Om^k, \, \nu \in \Omega^l. 
\end{align*}

Any fodc admits an extension to a differential calculus $\Omega^\bullet$ which is maximal in the sense that there exists a unique differential map from $\Omega^\bullet$ onto any other extension of $\Omega^1$, see \cite[\textsection 1.5]{BeggsMajid:Leabh} for details. We call this extension the {\em maximal prolongation} of the first-order calculus. 

A differential calculus $\Omega^\bullet$ over a left $A$-comodule algebra $P$ is said to be {\em covariant} if the coaction $\DEL_L:P \to A \otimes P$ extends to a (necessarily unique) $A$-comodule algebra structure $\DEL_L:\Om^\bullet \to A \otimes \Omega^\bullet$, with respect to which the differential $\exd$ is a left $A$-comodule map. Covariance for a right $A$-comodule algebra is defined analogously. The maximal prolongation of a covariant fodc is covariant. See \cite[\textsection 1]{BeggsMajid:Leabh} for a more detailed discussion of differential calculi.

\subsection{Complex Structures}

We now recall the definition of a complex structure as introduced in \cite{ BS,KLVSCP1}. This abstracts the properties of the de Rham complex of a classical complex manifold \cite{HUY}. 

\begin{defn}
A {\em complex structure} $\Om^{(\bullet,\bullet)}$ for a differential $*$-calculus $(\Om^{\bullet},\exd)$ is a choice of \mbox{$\mathbb{Z}_{\geq 0}^2$-algebra} grading $\bigoplus_{(a,b)\in \mathbb{Z}_{\geq 0}^2} \Om^{(a,b)}$ for $\Om^{\bullet}$ such that
\begin{align*}
1. \, \Om^k = \bigoplus_{a+b = k} \Om^{(a,b)}, & & 2.\, 
\big(\Om^{(a,b)}\big)^* = \Om^{(b,a)}, & & 3. \, \exd \Om^{(a,b)} \subseteq \Omega^{(a+1,b)} \oplus \Omega^{(a,b+1)},
\end{align*}
for all $k \in \mathbb{Z}_{\geq 0}$, and $(a,b) \in \mathbb{Z}^{2}_{\geq 0}$.
\end{defn}

We call an element of $\Om^{(a,b)}$ an $(a,b)$-form, and denote
\begin{align*}
\del|_{\Om^{(a,b)}} : = \proj_{\Om^{(a+1,b)}} \circ \exd, & & \ol{\del}|_{\Om^{(a,b)}} : = \proj_{\Om^{(a,b+1)}} \circ \exd,
\end{align*}
where $\proj_{\Om^{(a+1,b)}}$, and $\proj_{\Om^{(a,b+1)}}$, are the projections from $\Omega^{a+b+1}$ to $\Omega^{(a+1,b)}$, and $\Omega^{(a,b+1)}$, respectively. It follows directly from the definition of a complex structure that the triple 
$
(\Om^{(\bullet,\bullet)}, \del,\ol{\del})
$
is a double complex.  Moreover, both $\del$ and $\adel$ are $*$-maps and satisfy the graded Leibniz rule. The {\em opposite} complex structure of a complex structure $\Om^{(\bullet,\bullet)}$ is the \mbox{$\mathbb{Z}^2_{\geq 0}$}-algebra grading $\overline{\Om}^{(\bullet,\bullet)}$, defined by $\ol{\Om}^{(a,b)} := \Om^{(b,a)}$, for $(a,b) \in \mathbb{Z}_{\geq 0}^2$. 

For $\Om^\bullet$ a covariant differential $*$-calculus $\Om^\bullet$ over a left $A$-comodule algebra $P$, we say that a complex structure for $\Om^\bullet$ is {\em covariant} if the $\mathbb{Z}_{\geq 0}^2$-decomposition is a decomposition in $^A\mathrm{Mod}$, the category of left $A$-comodules. 

\subsection{Holomorphic Structures}

Building on this idea we define noncommutative holomorphic vector bundles via the classical Koszul--Malgrange characterisation of holomorphic bundles \cite{KM}

Any connection can be extended to a map $\nabla: \Omega^\bullet \otimes_B \mathcal{F} \to \Omega^\bullet \otimes_B \mathcal{F}$ uniquely defined by 
\begin{align*}
\nabla(\omega \otimes f) =  \exd \omega \otimes f + (-1)^{|\omega|} \, \omega \wedge \nabla f, & & \textrm{for } f \in \F, \, \omega \in \Omega^{\bullet},
\end{align*}
for a homogeneous form $\omega$ with degree $|\omega|$. The \emph{curvature} of a connection is the left $B$-module map $\nabla^2: \mathcal{F} \to \Omega^2 \otimes_B \mathcal{F}$. A connection is said to be {\em flat} if $\nabla^2 = 0$. Since $\nabla^2(\omega \otimes f) = \omega \wedge \nabla^2(f)$, a connection is flat if and only if the pair $(\Omega^\bullet \otimes_B \F, \nabla)$ is a complex. 

\begin{defn}
For an algebra $B$, a \emph{holomorphic vector bundle over $B$} is a pair $(\mathcal{F},\adel_{\mathcal{F}})$, where $\mathcal{F}$ is a finitely generated projective left $B$-module, and the $\adel_{\mathcal{F}}: \mathcal{F} \to \Omega^{(0,1)} \otimes_B \mathcal{F}$ is a flat $(0,1)$-connection, which we call the \emph{holomorphic structure} for $(\F, \adel_{\F})$. 
\end{defn}

The requirement of flatness means that any holomorphic vector bundle has an associated noncommutative Dolbeault cohomology $H^{\bullet}(\F)^{(\bullet,\bullet)}$. Moreover, as shown in \cite{OSV}, there is an associated noncommutative generalisation of the Kodaira vanishing theorem.

\subsection{The Irreducible Quantum Flag Manifolds}

A $*$-algebra structure on $U_q(\mathfrak{g})$, called the {\em compact real form}, is defined by
\begin{align*}
K^*_i : = K_i, & & E^*_i := K_i F_i, & & F^*_i := E_i K_i \inv,
\end{align*}
and gives $U_q(\mathfrak{g})$ the structure of a Hopf $*$-algebra. This $*$-structure dualises to a Hopf $*$-algebra structure on $\OO_q(G)$. The quantum flag manifolds $\OO_q(G/L_S)$ are $*$-subalgebras of $\OO_q(G/L_S)$  since $U_q(\mathfrak{l}_S)$ is a Hopf $*$-subalgebra of $U_q(\mathfrak{g})$. 

The \emph{Heckenberger--Kolb calculus} of $\OO_q(G/L_S)$ is the direct sum fodc 
$$
\Omega^{1}_q(G/L_S) := \Omega^{(1,0)}_q(G/L_S) \oplus \Omega^{(0,1)}_q(G/L_S).
$$ 
The maximal prolongation of $\Omega^{1}_q(G/L_S)$ is a covariant differential $*$-calculus of classical dimension. Moreover, it has a unique pair of opposite left $\OO_q(G)$-covariant complex structures extending the decomposition of $\Omega^1_q(G/L_S)$ into its $(1,0)$ and $(0,1)$-summands. For the quantum flag manifolds $\OO_q(G/L_S)$, and the associated line modules $\EE_k$, for $k \in \mathbb{Z}$, the left $\OO_q(G)$-covariant $(0,1)$-connections presented in \textsection \ref{subsection:HoloStructures} are in fact holomorphic structures \cite[\textsection 4]{DOKSS}. Thus they give a direct $q$-deformation of the classical holomoprhic structures of the line bundles over $G/L_S$. Moreover, the analogous results hold for the opposite complex structure and the $(1,0)$-connections $\del_{\EE_k}$.

\bibliographystyle{siam}

\end{document}